\documentclass[12pt,a4paper,leqno]{amsart} %Fr o m detta
\usepackage{amssymb}
\usepackage{amsthm}
\usepackage{latexsym}
\usepackage{amsmath}
\usepackage{mathrsfs}
\usepackage[X2,T1]{fontenc}
\usepackage{calc}                         %Detta behövs för att fixa
\usepackage[russian,english]{babel}
\usepackage{cite}

\newcommand{\scal}[2]{\langle #1,#2\rangle}

\newcommand{\rr}[1]{\mathbf R^{#1}}

\newcommand{\nm}[2]{\Vert #1\Vert _{#2}}
\newcommand{\nmm}[1]{\Vert #1\Vert }
\newcommand{\abp}[1]{\vert #1\vert}

\newcommand{\op}{\operatorname{Op}}

\newcommand{\fy}{\varphi}

\newcommand{\cdo}{\, \cdot \, }

\newcommand{\eabs}[1]{\langle #1\rangle}

\newcommand{\vrum}{\vspace{0.1cm}}
\newcommand{\ww}{{\text w}}
\newcommand{\wpr}{{\text{\footnotesize $\#$}}}
\newcommand{\essup}[1]{\underset{#1}{\operatorname {ess\, sup}}{\, }}

\numberwithin{equation}{section}          %Detta gör att man får
\newtheorem{thm}{Theorem}
\numberwithin{thm}{section}
\newtheorem*{tom}{\rubrik}
\newcommand{\rubrik}{}
\newtheorem{prop}[thm]{Proposition}
\newtheorem{cor}[thm]{Corollary}
\newtheorem{lemma}[thm]{Lemma}

\theoremstyle{definition}

\newtheorem{defn}[thm]{Definition}

\theoremstyle{remark}

\newtheorem{rem}[thm]{Remark}              %T o m hit är bara allmän
\author{Joachim Toft}

\address{Department of Mathematics and Systems Engineering,
V{\"a}xj{\"o} University, Sweden}

\email{joachim.toft@vxu.se}

\title{Twisted convolution, pseudo-differential operators and
Fourier modulation spaces}

\begin{document}

\begin{abstract}
We discuss continuity of the twisted convolution on (weighted) Fourier modulation spaces. We use these results to establish continuity results for the twisted convolution on Lebesgue spaces. For example we prove that if $\omega$ is an appropriate weight and $1\le p\le 2$, then $L^p_{(\omega )}$ is an algebra under the twisted convolution.

\par

We also discuss continuity for pseudo-differential operators with symbols in Fourier modulation spaces.
\end{abstract}

\par

\maketitle

%%%%%%%%%%%%%%%%%%%%%%%%%%%%%%%
\section{Introduction}\label{sec0}
%%%%%%%%%%%%%%%%%%%%%%%%%%%%%%%

\par

In this paper we continue the discussions from \cite{HTW} concerning
continuity and algebraic properties for pseudo-differential operators
in background of Lebesgue spaces and the theory of modulation
spaces. These investigations also involve studies of twisted
convolutions, which are connected to the pseudo-differential
calculus, in the sense that the Fourier transform of a Weyl product
is essentially a twisted convolution of the Weyl symbols. By combining
the latter property with continuity results of the Weyl product on
modulation spaces in \cite{HTW}, we establish continuity properties of
the twisted convolution on Fourier modulation spaces. From these
results we thereafter prove continuity properties of
the twisted convolution on weighted Fourier Lebesgue spaces.

\par

We also consider continuity for pseudo-differential operators
with symbols in Fourier modulation spaces. We establish continuity
properties of such pseudo-differential operators when acting between
modulation spaces and Fourier modulation spaces. These investigations
are based on an important result by Cordero and Okoudjou in \cite{CO1}
concerning mapping properties of short-time Fourier transforms on
modulation spaces.

\par

The (classical) modulation spaces $M^{p,q}$, $p,q \in [1,\infty]$, as
introduced by Feichtinger in \cite{Fei3},
consist of all tempered distributions whose
short-time Fourier transforms (STFT) have finite mixed $L^{p,q}$
norm. (Cf. \cite{Fei6} and the references therein for an updated describtion modulation spaces.)  It follows that the parameters $p$ and $q$ to some extent
quantify the degrees of asymptotic decay and singularity of
the distributions in $M^{p,q}$. The theory of modulation spaces was
developed further and generalized in
\cite{Fei4,Fei5,FG1,FG2,FG3,Groc1}, where
Feichtinger and Gr{\"o}chenig established the theory of coorbit
spaces. In particular, the modulation spaces $M^{p,q}_{(\omega )}$ and
$W^{p,q}_{(\omega )}$,
where $\omega$ denotes a weight function on
phase (or time-frequency shift) space, appears as the set of tempered
(ultra-) distributions
whose STFT belong to the weighted and mixed Lebesgue space
$L^{p,q}_{1,(\omega )}$ and $L^{p,q}_{2,(\omega )}$ respectively. (See
Section \ref{sec1} for strict definitions.) By choosing the weight
$\omega$ in appropriate ways, the space
$W^{p,q}_{(\omega )}$ becomes a Wiener amalgam space,
introduced in \cite {Fei1} by Feichtinger.

\par

A major idea behind the design of these spaces was to find
useful Banach spaces, which are defined in a way similar to Besov
and Triebel-Lizorkin spaces, in the sense of replacing the dyadic
decomposition on the Fourier transform side, characteristic to Besov
and Triebel-Lizorkin spaces, with a
\textit{uniform} decomposition. From the construction of these
spaces, it turns out that modulation spaces of the form $M^{p,q}_{(\omega )}$ and Besov spaces in some
sense are rather similar, and sharp embeddings between these
spaces can be found in \cite{To4, To6}, which
are improvements of certain embeddings in \cite {Grob}. (See
also \cite {ST1} for verification of the sharpness.) In the same
way it follows that modulation spaces of the form $W^{p,q}_{(\omega )}$ and Triebel-Lizorkin
spaces are rather similar.

\par

During the last 15 years many results have been proved which confirm
the usefulness of the modulation
spaces and their Fourier transforms in time-frequency analysis, where
they occur naturally. For example, in
\cite{FG3,Groc2,GL1}, it is shown that all
such spaces admit reconstructible sequence space representations
using Gabor frames.

\par

Parallel to this development, modulation spaces have been incorporated
into the calculus of pseudo-differential operators.
In fact, in \cite{Sj1}, Sj{\"o}strand introduced
the modulation space $M^{\infty,1}$, which contains non-smooth
functions, as a symbol class and proved that
$M^{\infty,1}$ corresponds to an algebra of operators which are
bounded on $L^2$.

\par

Gr{\"o}chenig and Heil thereafter proved in
\cite{GH1,Groc2} that each operator with symbol in
$M^{\infty,1}$ is continuous on all modulation spaces $M^{p,q}$, $p,q
\in [1,\infty]$. This extends Sj{\"o}strand's result since
$M^{2,2}=L^2$. Some generalizations to operators with
symbols in general unweighted modulation spaces were obtained in
\cite{GH2,To4}, and in \cite{To5, To7} some further
extensions  involving weighted
modulation spaces are presented. Modulation spaces in
pseudodifferential calculus is currently an active field of research
(see e.{\,}g.
\cite{Groc3, GH2, GH3, Her1, La1, La2,
PT2, ST1, Ta, Te1, To2, To4, To7}).

\par

In the Weyl calculus of pseudo-differential operators, operator
composition corresponds
on the symbol level to the Weyl product, sometimes also called the
twisted product, denoted by $\wpr$. A problem in this field is to find
conditions on the weight functions $\omega _j$ and
$p_j,q_j\in [1,\infty]$, that are necessary and sufficient for the map
\begin{equation}\label{weylmap}
\mathscr S(\rr {2d})\times \mathscr S(\rr {2d}) \ni (a_1,a_2)\mapsto
a_1\wpr a_2 \in \mathscr S(\rr {2d})
\end{equation}
to be uniquely extendable to a map from $M_{(\omega _1)}^{p_1,q_1}(\rr
{2d}) \times M_{(\omega _2)}^{p_2,q_2}(\rr {2d})$ to $M_{(\omega
_0)}^{p_0,q_0}(\rr {2d})$,
which is continuous in the sense that for some constant $C>0$ it holds
\begin{equation}\label{holderyoung3}
\| a_1 \wpr a_2 \|_{M_{(\omega _0)}^{p_0,q_0} }
\leq C \| a_1 \|_{M_{(\omega _1)}^{p_1,q_1} } \| a_2 \|_{M_{(\omega
_2)}^{p_2,q_2} },
\end{equation}
when $a_1 \in M_{(\omega _1)}^{p_1,q_1}(\rr {2d})$ and $a_2 \in
M_{(\omega _2)}^{p_2,q_2}(\rr {2d})$.
Important contributions in this context can be found in
\cite{Groc3,HTW,La1,Sj1, To2}, where Theorem 0.3$'$ in
\cite{HTW} seems to be the most general result so far.

\par

The Weyl product on the Fourier transform side
is given by a twisted convolution, $*_\sigma$. It follows that the
continuity questions here above are the same as finding appropriate
conditions on $\omega _j$ and $p_j,q_j\in [1,\infty]$, in order for
the map
\begin{equation}\label{twistmap}
\mathscr S(\rr {2d})\times \mathscr S(\rr {2d}) \ni (a_1,a_2)\mapsto
a_1*_\sigma a_2 \in \mathscr S(\rr {2d})
\end{equation}
to be uniquely extendable to a map from $W_{(\omega _1)}^{p_1,q_1}(\rr
{2d}) \times W_{(\omega _2)}^{p_2,q_2}(\rr {2d})$ to $W_{(\omega
_0)}^{p_0,q_0}(\rr {2d})$, which is continuous in the sense that for
some constant $C>0$ it holds
\begin{equation}\label{holderyoung4}
\| a_1 *_\sigma a_2 \|_{W_{(\omega _0)}^{p_0,q_0} }
\leq C \| a_1 \|_{W_{(\omega _1)}^{p_1,q_1} } \| a_2 \|_{W_{(\omega
_2)}^{p_2,q_2} },
\end{equation}
when $a_1 \in W_{(\omega _1)}^{p_1,q_1}(\rr {2d})$ and $a_2 \in
W_{(\omega _2)}^{p_2,q_2}(\rr {2d})$. In this context the continuity
result which corresponds to Theorem 0.3$'$ in \cite{HTW} is Theorem
\ref{algthm2} in Section \ref{sec2}.

\par

In the end of Section \ref{sec2} we especially consider the case when
$p_j=q_j=2$. In this case, $W_{(\omega _j)}^{2,2}$ agrees with
$L^2_{(\omega _j)}$, for appropriate choices of $\omega _j$. Hence,
for such $\omega _j$, it
follows immediately from Theorem \ref{algthm2} that the map
\eqref{twistmap} extends to a continuous mapping from $L^2_{(\omega
_1)}(\rr {2d})\times L^2_{(\omega _2)}(\rr {2d})$ to $L^2_{(\omega
_0)}(\rr {2d})$, and that
$$
\nm {a_1*_\sigma a_2}{L^2_{(\omega _0)}}\le C\nm {a_1}{L^2_{(\omega
_1)}}\nm {a_2}{L^2_{(\omega _2)}},
$$
when $a_1\in L^2_{(\omega _1)}(\rr {2d})$ and $a_2\in L^2_{(\omega
_2)}(\rr {2d})$. In Section \ref{sec2} we prove a more general result,
by combining this result with Young's inequality, and then using
interpolation. Finally we use these results in Section \ref{sec2.5} to enlarge the class of possible window functions in the definition of modulation space norm.

\par

%%%%%%%%%%%%%%%%%%%%%%%%%%%%%%%%%%%%
\section{Preliminaries}\label{sec1}
%%%%%%%%%%%%%%%%%%%%%%%%%%%%%%%%%%%%

\par

In this section we recall some notations and basic results. The proofs
are in general omitted.

\par

We start by discussing appropriate conditions for the involved weight
functions. Assume that $\omega$ and $v$ are positive and measureable
functions on $\rr d$. Then $\omega$ is called $v$-moderate if
\begin{equation}\label{moderate}
\omega (x+y) \leq C\omega (x)v(y)
\end{equation}
for some constant $C$ which is independent of $x,y\in \rr d$. If $v$
in \eqref{moderate} can be chosen as a polynomial, then $\omega$ is
called polynomially moderated. We let $\mathscr P(\rr d)$ be the set
of all polynomially moderated functions on $\rr d$. If $\omega (x,\xi
)\in \mathscr P(\rr {2d})$ is constant with respect to the
$x$-variable ($\xi$-variable), then we sometimes write $\omega (\xi )$
($\omega (x)$) instead of $\omega (x,\xi )$. In this case we consider
$\omega$ as an element in $\mathscr P(\rr {2d})$ or in $\mathscr P(\rr
d)$ depending on the situation.

\medspace

The Fourier transform $\mathscr F$ is the linear and continuous
mapping on $\mathscr S'(\rr d)$ which takes the form
$$
(\mathscr Ff)(\xi )= \widehat f(\xi ) \equiv (2\pi )^{-d/2}\int _{\rr
{d}} f(x)e^{-i\scal  x\xi }\, dx
$$
when $f\in L^1(\rr d)$. We recall that $\mathscr F$ is a homeomorphism
on $\mathscr S'(\rr d)$ which restricts to a homeomorphism on $\mathscr
S(\rr d)$ and to a unitary operator on $L^2(\rr d)$.

\par

Let $\fy \in \mathscr S'(\rr d)$ be fixed, and let $f\in \mathscr
S'(\rr d)$. Then the short-time Fourier transform $V_\fy f(x,\xi )$ of
$f$ with respect to the \emph{window function} $\fy$ is the tempered
distribution on $\rr {2d}$ which is defined by
$$
V_\fy f(x,\xi ) \equiv \mathscr F(f \, \overline {\fy (\cdo -x)}(\xi ).
$$
If $f ,\fy \in \mathscr S(\rr d)$, then it follows that
$$
V_\fy f(x,\xi ) = (2\pi )^{-d/2}\int f(y)\overline {\fy
(y-x)}e^{-i\scal y\xi}\, dy .
$$

\medspace

Next we recall some properties on modulation spaces and their Fourier
transforms. Assume that
$\omega \in \mathscr P(\rr {2d})$ and that $p,q\in [1,\infty ]$. Then
the mixed Lebesgue space $L^{p,q}_{1,(\omega )}(\rr {2d})$ consists of all
$F\in L^1_{loc}(\rr {2d})$ such that $\nm F{L^{p,q}_{1,(\omega )}}<\infty$, and
$L^{p,q}_{2,(\omega )}(\rr {2d})$ consists of all
$F\in L^1_{loc}(\rr {2d})$ such that $\nm F{L^{p,q}_{2,(\omega )}}<\infty$. Here
\begin{align*}
\nm F{L^{p,q}_{1,(\omega )}} &= \Big (\int \Big (\int |F(x,\xi )\omega
(x,\xi )|^p\, dx\Big )^{q/p}\, d\xi \Big )^{1/q},
\intertext{and}
\nm F{L^{p,q}_{2,(\omega )}} &= \Big (\int \Big (\int |F(x,\xi )\omega
(x,\xi )|^q\, d\xi \Big )^{p/q}\, dx \Big )^{1/p},
\end{align*}
with obvious modifications when $p=\infty$ or $q=\infty$.

\par

Assume that $p,q\in [1,\infty ]$, $\omega \in \mathscr P(\rr
{2d})$ and $\fy \in \mathscr S(\rr d)\setminus 0$ are fixed. Then the
\emph{modulation space} $M^{p,q}_{(\omega )}(\rr d)$ is the Banach
space which consists of all $f\in \mathscr S'(\rr d)$ such that
\begin{equation}\label{modnorm}
\nm f{M^{p,q}_{(\omega )}}\equiv \nm {V_\fy f}{L^{p,q}_{1,(\omega
)}}<\infty .
\end{equation}
The modulation space $W^{p,q}_{(\omega )}(\rr d)$ is the Banach space
which consists of all $f\in \mathscr S'(\rr d)$ such that
\begin{equation}\label{fourmodnorm}
\nm f{W^{p,q}_{(\omega )}}\equiv \nm {V_\fy f}{L^{p,q}_{2,(\omega
)}}<\infty .
\end{equation}
The definitions of $M^{p,q}_{(\omega )}(\rr d)$ and
$W^{p,q}_{(\omega )}(\rr d)$ are independent of the choice of $\fy$
and different $\fy$ gives rise to equivalent norms. (See Proposition
\ref{p1.4} below). From the fact that
$$
V_{\widehat \fy}\widehat f (\xi ,-x) =e^{i\scal x\xi}V_{\check
\fy}f(x,\xi ),\qquad \check \fy (x)=\fy (-x),
$$
it follows that
$$
f\in W^{p,q}_{(\omega )}(\rr d)\quad \Longleftrightarrow \quad \widehat f\in
M^{p,q}_{(\omega _0)}(\rr d),\qquad \omega _0(\xi ,-x)=\omega (x,\xi ).
$$

\par

For conveniency we set $M^p _{(\omega )}= M^{p,p}_{(\omega
)}$, which agrees with $W^p_{(\omega )}=W^{p,p}_{(\omega
)}$. Furthermore we set $M^{p,q}=M^{p,q}_{(\omega )}$ and
$W^{p,q}=W^{p,q}_{(\omega )}$ if $\omega \equiv 1$. If $\omega$ is
given by $\omega (x,\xi )=\omega _1(x)\omega _2(\xi )$, for some
$\omega _1,\omega _2\in \mathscr P(\rr d)$, then $W^{p,q}_{(\omega )}$
is a Wiener amalgam space, introduced by Feichtinger in \cite{Fei1}.

\par

The proof of the following proposition is omitted, since the results
can be found in \cite {Fei2,Fei3,FG1,FG2,FG3, Groc2,To4,
To5, To6, To7}. Here and in what follows, $p'\in[1,\infty]$ denotes the
conjugate exponent of $p\in[1,\infty]$, i.{\,}e. $1/p+1/p'=1$ should
be fulfilled.

\par

\begin{prop}\label{p1.4}
Assume that $p,q,p_j,q_j\in [1,\infty ]$ for $j=1,2$, and $\omega
,\omega _1,\omega _2,v\in \mathscr P(\rr {2d})$ are such that $\omega$
is $v$-moderate and $\omega _2\le C\omega _1$ for some constant
$C>0$. Then the following are true:
\begin{enumerate}
\item[{\rm{(1)}}] if $\fy \in M^1_{(v)}(\rr d)\setminus 0$, then $f\in
M^{p,q}_{(\omega )}(\rr d)$ if and only if \eqref{modnorm} holds,
i.{\,}e. $M^{p,q}_{(\omega )}(\rr d)$. Moreover, $M^{p,q}_{(\omega )}$
is a Banach space under the norm in \eqref{modnorm} and different
choices of $\fy$ give rise to equivalent norms;

\vrum

\item[{\rm{(2)}}] if  $p_1\le p_2$ and $q_1\le q_2$  then
$$
\mathscr S(\rr d)\hookrightarrow M^{p_1,q_1}_{(\omega _1)}(\rr
n)\hookrightarrow M^{p_2,q_2}_{(\omega _2)}(\rr d)\hookrightarrow
\mathscr S'(\rr d)\text ;
$$

\vrum

\item[{\rm{(3)}}] the $L^2$ product $( \cdo ,\cdo )$ on $\mathscr
S$ extends to a continuous map from $M^{p,q}_{(\omega )}(\rr
n)\times M^{p'\! ,q'}_{(1/\omega )}(\rr d)$ to $\mathbf C$. On the
other hand, if $\nmm a = \sup \abp {(a,b)}$, where the supremum is
taken over all $b\in \mathscr {S}(\rr d)$ such that
$\nm b{M^{p',q'}_{(1/\omega )}}\le 1$, then $\nmm {\cdot}$ and $\nm
\cdot {M^{p,q}_{(\omega )}}$ are equivalent norms;

\vrum

\item[{\rm{(4)}}] if $p,q<\infty$, then $\mathscr S(\rr d)$ is dense in
$M^{p,q}_{(\omega )}(\rr d)$ and the dual space of $M^{p,q}_{(\omega
)}(\rr d)$ can be identified
with $M^{p'\! ,q'}_{(1/\omega )}(\rr d)$, through the form $(\cdo  ,\cdo
)_{L^2}$. Moreover, $\mathscr S(\rr d)$ is weakly dense in $M^{\infty
}_{(\omega )}(\rr d)$.
\end{enumerate}
Similar facts hold if the $M^{p,q}_{(\omega )}$ spaces are replaced by
$W^{p,q}_{(\omega )}$ spaces.
\end{prop}

\par

Proposition \ref{p1.4}{\,}(1) allows us  be rather vague concerning
the choice of $\fy \in  M^1_{(v)}\setminus 0$ in
\eqref{modnorm} and \eqref{fourmodnorm}. For example, if $C>0$ is a
constant and $\mathscr A$ is a subset of $\mathscr S'$, then $\nm
a{W^{p,q}_{(\omega )}}\le C$ for
every $a\in \mathscr A$, means that the inequality holds for some choice
of $\fy \in  M^1_{(v)}\setminus 0$ and every $a\in
\mathscr A$. Evidently, a similar inequality is true for any other choice
of $\fy \in  M^1_{(v)}\setminus 0$, with  a suitable constant, larger
than $C$ if necessary.

\par

In the following remark we list some other properties for modulation
spaces. Here and in what follows we let
$\eabs x =(1+|x|^2)^{1/2}$, when $x\in \rr d$.

\par

\begin{rem}\label{p1.7}
Assume that $p,p_1,p_2,q,q_1,q_2\in [1,\infty ]$ are such that
$$
q_1\le\min (p,p'),\quad q_2\ge \max (p,p'),\quad  p_1\le\min
(q,q'),\quad p_2\ge \max (q,q'),
$$
and that $\omega, v \in
\mathscr P(\rr {2d})$ are such that $\omega$ is $v$-moderate. Then the
following is true:
\begin{enumerate}
\item[{\rm{(1)}}]  if $p\le q$, then $W^{p,q}_{(\omega )}(\rr
d)\subseteq M^{p,q}_{(\omega )}(\rr d)$, and if $p\ge q$, then
$M^{p,q}_{(\omega )}(\rr d)\subseteq W^{p,q}_{(\omega )}(\rr
d)$. Furthermore, if $\omega (x,\xi )=\omega (x)$, then
$$
M^{p,q_1}_{(\omega )}(\rr d)\subseteq W^{p,q_1}_{(\omega )}(\rr d)
\subseteq L^p_{(\omega )}(\rr d)\subseteq W^{p,q_2}_{(\omega )}(\rr
d)\subseteq M^{p,q_2}_{(\omega )}(\rr d).
$$
In particular, $M^2_{(\omega )}=W^2_{(\omega )}=L^2_{(\omega )}$. If
instead $\omega (x,\xi )=\omega (\xi )$, then
$$
W^{p_1,q}_{(\omega )}(\rr d)\subseteq M^{p,q_1}_{(\omega )}(\rr d)
\subseteq \mathscr FL^q_{(\omega )}(\rr d)\subseteq M^{p_2,q}_{(\omega
)}(\rr d)\subseteq W^{p_2,q}_{(\omega )}(\rr d).
$$
Here $\mathscr FL^q_{(\omega _0)}(\rr d)$ consists of all $f\in
\mathscr S'(\rr d)$ such that
$$
\nm {\widehat f \, \omega _0}{L^q}<\infty ;
$$

\vrum

\item[{\rm{(2)}}] if $\omega (x,\xi )=\omega (x)$, then the following
conditions are equivalent:
\begin{itemize}
\item $M^{p,q}_{(\omega )}(\rr d)\subseteq C(\rr d)$;

\vrum

\item $W^{p,q}_{(\omega )}(\rr d)\subseteq C(\rr d)$;

\vrum

\item $q=1$.
\end{itemize}

\vrum

\item[{\rm{(3)}}] $M^{1,\infty}(\rr d)$ and $W^{1,\infty}(\rr d)$ are
convolution algebras. If $C_B'(\rr d)$ is the set of all measures on
$\rr d$ with bounded mass, then
$$
W^{1,\infty}(\rr d)\subseteq C_B'(\rr d)\subseteq M^{1,\infty}(\rr
d)\text ;
$$

\vrum

\item[{\rm{(4)}}] if $x_0\in \rr d$ is fixed and $\omega _0(\xi
)=\omega (x_0,\xi )$, then
$$
M^{p,q}_{(\omega )}\cap \mathscr E' =W^{p,q}_{(\omega )}\cap \mathscr
E' =\mathscr  FL^q_{(\omega _0)}\cap \mathscr E' \text ;
$$

\vrum

\item[{\rm{(5)}}] if $\omega (x,\xi )=\omega _0(\xi ,-x)$, then the
Fourier transform on $\mathscr S'(\rr d)$ restricts to a homeomorphism
from $M^{p}_{(\omega )}(\rr d)$ to $M^{p}_{(\omega _0 )}(\rr d)$. In
particular, if $\omega =\omega _0$, then $M^p_{(\omega )}$ is
invariant under the Fourier transform. Similar facts hold for partial
Fourier transforms;

\vrum

\item[{\rm{(6)}}] for each $x,\xi\in \rr d$ we have
\begin{align*}
\nm {e^{i\scal\cdo \xi}f(\cdo -x)}{M^{p,q}_{(\omega)}}\le C
v(x,\xi) \nm f{M^{p,q}_{(\omega)}},
\intertext{and}
\nm {e^{i\scal\cdo \xi}f(\cdo -x)}{W^{p,q}_{(\omega)}}\le C v(x,\xi)
\nm f{W^{p,q}_{(\omega)}}
\end{align*}
for some constant $C$ which is independent of $f\in \mathscr S'(\rr
d)$;

\vrum

\item[{\rm{(7)}}] if $\tilde {\omega}(x,\xi)=\omega(x,-\xi)$ then
$f\in M^{p,q}_{(\omega)}$ if and only if $\overline f\in
M^{p,q}_{(\tilde {\omega})}$;

\vrum

\item[{\rm{(8)}}] if $s\in \mathbf R$ and $\omega (x,\xi )=\eabs \xi
^s$, then $M^2_{(\omega )}=W^2_{(\omega )}$ agrees with $H^2_s$, the
Sobolev space of distributions with $s$ derivatives in $L^2$. That is,
$H^2_s$ consists of all $f\in \mathscr S'$ such that $\mathscr
F^{-1}(\eabs \cdo ^s\widehat f)\in L^2$.
\end{enumerate}
(See e.{\,}g. \cite{Fei2,Fei3,FG1,FG2,FG3, Groc2,To4,
To5, To6, To7}.)
\end{rem}

\par

\medspace

We also need some facts in Section 2 in \cite{To7} on narrow
convergence. For any $f\in \mathscr S'(\rr d)$, $\omega \in \mathscr
P(\rr {2d})$, $\fy \in \mathscr S(\rr d)$ and $p\in [1,\infty ]$, we
set
$$
H_{f,\omega ,p}(\xi )=\Big ( \int _{\rr d}\abp {V_\fy f(x,\xi )\omega
(x,\xi )}^p\, dx\Big )^{1/p}.
$$

\par

\begin{defn}\label{p2.1}
Assume that $f,f_j\in M^{p,q}_{(\omega )}(\rr {d} )$,
$j=1,2,\dots \ $. Then $f_j$ is said to converge {\it narrowly} to $f$
(with respect to $p,q\in [1,\infty ]$, $\fy \in \mathscr S(\rr
d)\setminus 0$ and $\omega \in \mathscr P(\rr {2d})$), if the following
conditions are satisfied:
\begin{enumerate}
\item $f_j\to f$ in $\mathscr S'(\rr {d} )$ as $j$ turns to $\infty$;

\par

\item $H_{f_j,\omega ,p}(\xi )\to H_{f,\omega ,p}(\xi )$ in $L^q(\rr
d)$ as $j$ turns to $\infty$.
\end{enumerate}
\end{defn}

\par

\begin{rem}\label{p2.2}
Assume that $f,f_1,f_2,\dots \in \mathscr S'(\rr
d)$ satisfies (1) in Definition \ref{p2.1}, and assume that $\xi \in
\rr d$. Then it follows from Fatou's lemma that
$$
\liminf _{j\to \infty}H_{f_j,\omega ,p}(\xi )\ge H_{f,\omega ,p}(\xi
)\quad \text{and}\quad \liminf _{j\to \infty}\nm
{f_j}{M^{p,q}_{(\omega )}}\ge \nm {f}{M^{p,q}_{(\omega )}}.
$$
\end{rem}

\par

The following proposition is important to us later on. We omit the
proof since the result is a restatement of Proposition 2.3 in \cite
{To7}.

\par

\begin{prop}\label{p2.3}
Assume that $p,q\in [1,\infty ]$ with  $q<\infty$ and that $\omega \in
\mathscr P(\rr {2d})$. Then $C_0^\infty (\rr d)$ is dense in
$M^{p,q}_{(\omega )}(\rr d)$ with respect to the narrow convergence.
\end{prop}

\medspace

Next we recall some facts in Chapter XVIII in \cite{Ho1} concerning
pseudo-differential operators. Assume that $a\in \mathscr 
S(\rr {2d})$, and that $t\in \mathbf R$ is fixed. Then the
pseudo-differential operator $a_t(x,D)$ in
\begin{equation}\label{e0.5}
\begin{aligned}
(a_t(x,D)f)(x &)
=
(\op _t(a)f)(x)
\\
=
(2\pi  &) ^{-d}\iint a((1-t)x+ty,\xi )f(y)e^{i\scal {x-y}\xi }\,
dyd\xi .
\end{aligned}
\end{equation}
is a linear and continuous operator on $\mathscr S(\rr d)$. For
general $a\in \mathscr S'(\rr {2d})$, the
pseudo-differential operator $a_t(x,D)$ is defined as the continuous
operator from $\mathscr S(\rr d)$ to $\mathscr S'(\rr d)$ with
distribution kernel
\begin{equation}\label{weylkernel}
K_{t,a}(x,y)=(2\pi )^{-n/2}(\mathscr F_2^{-1}a)((1-t)x+ty,y-x),
\end{equation}
where $\mathscr F_2F$ is the partial
Fourier transform of $F(x,y)\in \mathscr S'(\rr{2d})$ with respect to
the $y$-variable. This definition makes sense, since
the mappings $\mathscr F_2$ and $F(x,y)\mapsto F((1-t)x+ty,y-x)$ are
homeomorphisms on $\mathscr S'(\rr {2d})$. Moreover, it agrees with
the operator in \eqref{e0.5} when $a\in \mathscr S(\rr {2d})$.
If $t=0$, then $a_t(x,D)$ agrees with
the Kohn-Nirenberg representation $a(x,D)$. If instead $t=1/2$, then
$a_t(x,D)$ is the Weyl operator $a^w(x,D)$ of $a$, and if $a,b\in
\mathscr S(\rr {2d})$, then the Weyl product $a\wpr b$ between $a$
and $b$ is the function which fulfills $(a\wpr b)^w(x,D) =
a^w(x,D)b^w(x,D)$.

\medspace

Next we recall the definition of symplectic Fourier transform,
twisted convolution and related objects. The even-dimensional vector
space $\rr {2d}$ is a (real) symplectic vector space with the
(standard) symplectic form
$$
\sigma(X,Y) = \sigma\big( (x,\xi ); (y,\eta ) \big) = \scal y \xi -
\scal x \eta
$$
where $\langle \cdot,\cdot \rangle$ denotes the usual scalar product
on $\rr d$.

\par

The symplectic Fourier transform for $a \in \mathscr{S}(\rr {2d})$
is defined by the formula
\begin{equation*}
(\mathscr{F}_\sigma a) (X)
= \widehat{a}(X)
= \pi^{-d}\int a(Y) e^{2 i \sigma(X,Y)}\,  dY.
\end{equation*}
Then $\mathscr{F}_\sigma^{-1} = {\mathscr{F}_\sigma}$ is continuous on
$\mathscr{S}(\rr {2d})$, and extends as
usual to a homeomorphism on $\mathscr{S}'(\rr {2d})$, and to a
unitary map on $L^2(\rr {2d})$. The symplectic short-time Fourier
transform of $a \in \mathscr{S}'(\rr {2d})$ with respect to the
window function $\fy \in \mathscr{S'}(\rr {2d})$ is defined
by
\begin{equation}\nonumber
\mathcal V_{\fy} a(X,Y) = \mathscr{F}_\sigma \big( a\, \fy (\cdo -X)
\big) (Y), \ X,Y \in \rr {2d}.
\end{equation}

\par

Assume that $\omega \in \mathscr P(\rr {4d})$. Then we let
$\mathcal M^{p,q}_{(\omega )}(\rr {2d})$ and $\mathcal
W^{p,q}_{(\omega )}(\rr {2d})$ denote the modulation spaces, where
the symplectic short-time Fourier
transform is used instead of the usual short-time Fourier transform in
the definitions of the norms. It follows that any property valid for
$M^{p,q}_{(\omega )}(\rr {2d})$ or $W^{p,q}_{(\omega )}(\rr {2d})$
carry over to $\mathcal M^{p,q}_{(\omega )}(\rr {2d})$ and $\mathcal
W^{p,q}_{(\omega )}(\rr {2d})$ respectively. For example, for the
symplectic short-time Fourier transform we have
\begin{equation}\label{stftsymplfour}
\mathcal V_{\mathscr F_\sigma \fy}(\mathscr F_\sigma a)(X,Y) =
e^{2i\sigma (Y,X)}\mathcal V_\fy a(Y,X),
\end{equation}
which implies that
\begin{equation}\label{twistfourmod}
\mathscr F_\sigma \mathcal M^{p,q}_{(\omega )}(\rr {2d}) = \mathcal
W^{q,p}_{(\omega _0)}(\rr {2d}), \qquad \omega _0(X,Y)=\omega (Y,X).
\end{equation}

\par

Assume that $a,b\in \mathscr S(\rr {2d})$. Then the twisted
convolution of $a$ and $b$ is defined by the formula
\begin{equation}\label{twist1}
(a \ast _\sigma b) (X)
= (2/\pi)^{d/2} \int a(X-Y) b(Y) e^{2 i \sigma(X,Y)}\, dY.
\end{equation}
The definition of $*_\sigma$ extends in different ways. For example,
it extends to a continuous multiplication on $L^p(\rr {2d})$ when $p\in
[1,2]$, and to a continuous map from $\mathscr S'(\rr {2d})\times
\mathscr S(\rr {2d})$ to $\mathscr S'(\rr {2d})$. If $a,b \in
\mathscr{S}'(\rr {2d})$, then $a \wpr b$ makes sense if and only if $a
*_\sigma \widehat b$ makes sense, and then
\begin{equation}\label{tvist1}
a \wpr b = (2\pi)^{-d/2} a \ast_\sigma (\mathscr F_\sigma {b}).
\end{equation}
We also remark that for the twisted convolution we have
\begin{equation}\label{weylfourier1}
\mathscr F_\sigma (a *_\sigma b) = (\mathscr F_\sigma a) *_\sigma b =
\check{a} *_\sigma (\mathscr F_\sigma b),
\end{equation}
where $\check{a}(X)=a(-X)$ (cf. \cite{To1,To2,To3}). A
combination of \eqref{tvist1} and \eqref{weylfourier1} give
\begin{equation}\label{weyltwist2}
\mathscr F_\sigma (a\wpr b) = (2\pi )^{-d/2}(\mathscr F_\sigma
a)*_\sigma (\mathscr F_\sigma b).
\end{equation}

\par

%%%%%%%%%%%%%%%%%%%%%%%%%%%%%%%
\section{Twisted convolution on modulation spaces and Lebesgue
spaces}\label{sec2}
%%%%%%%%%%%%%%%%%%%%%%%%%%%%%%%

\par

In this section we discuss algebraic properties of the twisted
convolution when acting on modulation spaces of the form
$W^{p,q}_{(\omega )}$. The most general result is equivalent to
Theorem 0.3$'$ in \cite{HTW}, which concerns continuity for
the Weyl product on modulation spaces. Thereafter we use this
result to establish continuity properties for the twisted convolution
when acting on weighted Lebesgue spaces.

\par

The following lemma is important in our investigations. The
proof is omitted since the result is an immediate consequence of Lemma
4.4 in \cite{To2} and its proof, \eqref{stftsymplfour},
\eqref{tvist1} and \eqref{weylfourier1}.

\par

\begin{lemma}\label{cornerstone}
Assume that $a_1\in \mathscr S'(\rr {2d})$, $a_2 \in \mathscr
S(\rr {2d})$ and $\fy _1,\fy _2 \in \mathscr S(\rr {2d})$. Then the
following is true:
\begin{enumerate}
\item if $\fy = \pi^d \fy _1 \wpr \fy _2$, then $\fy \in \mathscr
S(\rr {2d})$, the map
$$
Z\mapsto e^{2 i \sigma(Z,Y)} (\mathcal V_{\chi_1} a_1) (X-Y+Z,Z) \,
(\mathcal V_{\chi_2} a_2)(X+Z,Y-Z)
$$
belongs to $L^1(\rr {2d})$, and
\begin{multline}\label{weylp1}
\mathcal V_{\fy } (a_1 \wpr a_2) (X,Y)
\\[1ex]
=
\int  e^{2 i \sigma(Z,Y)} (\mathcal V_{\chi_1} a_1)(X-Y+Z,Z) \,
(\mathcal V_{\chi_2}a_2)
(X+Z,Y-Z) \, dZ\text ;
\end{multline}

\vrum

\item if $\fy = 2^{-d} \fy _1 *_\sigma \fy _2$, then $\fy \in \mathscr
S(\rr {2d})$, the map
$$
Z\mapsto e^{2 i \sigma(X,Z-Y)} (\mathcal V_{\chi_1} a_1) (X-Y+Z,Z) \,
(\mathcal V_{\chi_2} a_2)(Y-Z,X+Z)
$$
belongs to $L^1(\rr {2d})$, and
\begin{multline}\label{twistp1}
\mathcal V_{\fy } (a_1 *_\sigma a_2) (X,Y)
\\[1ex]
=
\int  e^{2 i \sigma(X,Z-Y)} (\mathcal V_{\chi_1} a_1)(X-Y+Z,Z) \,
(\mathcal V_{\chi_2}a_2)(Y-Z,X+Z) \, dZ
\end{multline}
\end{enumerate}
\end{lemma}

\par

The first part of the latter result is used in \cite{HTW} to prove the
following result, which is essentially a restatement of Theorem 0.3$'$
in \cite{HTW}. Here we assume that the involved weight functions
satisfy
\begin{equation}\label{vikt1}
\omega_0(X,Y) \leq C \omega_1(X-Y+Z,Z)
\omega_2(X+Z,Y-Z),\quad \! \! X,Y,Z\in \rr {2d}.
\end{equation}
for some constant $C>0$, and that $p_j, q_j \in
[1,\infty]$ satisfy 
\begin{align}
\frac {1}{p_1}+\frac {1}{p_2}-\frac {1}{p_0} &= 1-\Big (\frac
{1}{q_1}+\frac {1}{q_2}-\frac {1}{q_0}\Big )\label{pqformulas1}
\intertext{and}
0\le \frac {1}{p_1}+\frac {1}{p_2}-\frac {1}{p_0}&\le \frac
{1}{p_j},\frac {1}{q_j}\le \frac {1}{q_1}+\frac {1}{q_2}-\frac
{1}{q_0} ,\quad j=0,1,2.\label{pqformulas2}
\end{align}

\par

\begin{thm}\label{algthm1}
Assume that $\omega _0,\omega _1,\omega _2\in
\mathscr P(\rr {4d})$ satisfy \eqref{vikt1}, and that $p_j, q_j \in
[1,\infty]$ for $j=0,1,2$, satisfy \eqref{pqformulas1} and
\eqref{pqformulas2}.
Then the map \eqref{weylmap} on  $\mathscr S(\rr{2d})$ extends
uniquely to a continuous map from $M_{(\omega _1)}^{p_1,q_1}(\rr
{2d}) \times M_{(\omega _2)}^{p_2,q_2}(\rr {2d})$ to $M_{(\omega _0)}
^{p_0,q_0}(\rr {2d})$, and for some constant $C>0$, the bound
\eqref{holderyoung3} holds for every $a_1\in M_{(\omega
_1)}^{p_1,q_1}(\rr {2d})$ and $a_2\in M_{(\omega _2)}^{p_2,q_2}(\rr
{2d})$.
\end{thm}

\par

The next result is an immediate consequence of  \eqref{twistfourmod},
\eqref{weyltwist2} and Theorem \ref{algthm1}. Here the condition
\eqref{vikt1} should be replaced by
\begin{equation}\label{vikt2}
\omega_0(X,Y) \leq C \omega_1(X-Y+Z,Z)
\omega_2(Y-Z,X+Z),\quad \! \! X,Y,Z\in \rr {2d}.
\end{equation}
and the condition \eqref{pqformulas2} should be replaced by
\begin{equation}\label{pqformulas3}
0\le \frac {1}{q_1}+\frac {1}{q_2}-\frac
{1}{q_0}\le \frac {1}{p_j},\frac {1}{q_j}\le \frac {1}{p_1}+\frac
{1}{p_2}-\frac {1}{p_0},\quad j=0,1,2.
\end{equation}

\par

\begin{thm}\label{algthm2}
Assume that $\omega _0,\omega _1,\omega _2\in
\mathscr P(\rr {4d})$ satisfy \eqref{vikt2}, and that $p_j, q_j \in
[1,\infty]$ for $j=0,1,2$, satisfy \eqref{pqformulas1} and
\eqref{pqformulas3}.
Then the map \eqref{twistmap} on $\mathscr S(\rr {2d})$ extends
uniquely to a continuous map from $W_{(\omega _1)}^{p_1,q_1}(\rr
{2d}) \times W_{(\omega _2)}^{p_2,q_2}(\rr {2d})$ to $W_{(\omega _0)}
^{p_0,q_0}(\rr {2d})$, and for some constant $C>0$, the bound
\eqref{holderyoung4} holds for every $a_1\in W_{(\omega
_1)}^{p_1,q_1}(\rr {2d})$ and $a_2\in W_{(\omega _2)}^{p_2,q_2}(\rr
{2d})$.
\end{thm}

\par

By using Theorem \ref{algthm2} we may generalize Proposition 1.4 in
\cite{To25} to involve continuity of the twisted convolution on
weighted Lebesgue spaces.

\par

\begin{thm}\label{twistedleb}
Assume that $\omega _0,\omega _1,\omega _2\in
\mathscr P(\rr {2d})$ and $p,p_1,p_2 \in [1,\infty]$ satisfy
\begin{align*}
\omega _0(X_1+X_2) &\le C\omega _1(X_1)\omega _2(X_2),\quad
p_1,p_2\le p
\\[1ex]
\text{and} \quad \max \Big ( \frac 1p,\frac 1{p'}\Big )  &\le \frac
1{p_1}+\frac 1{p_2}-\frac 1{p}\le 1,
\end{align*}
for some constant $C$. Then the map \eqref{twistmap} extends
uniquely to a continuous mapping from $L^{p_1}_{(\omega _1)}(\rr
{2d})\times L^{p_2}_{(\omega _2)}(\rr {2d})$ to $L^p_{(\omega _0)}(\rr
{2d})$. Furthermore, for some constant $C$ it holds
\begin{multline*}
\nm {a_1*_\sigma a_2}{L^p_{(\omega _0)}} \le C \nm
{a_1}{L^{p_1}_{(\omega _1)}} \nm {a_2}{L^{p_2}_{(\omega _2)}},
\\[1ex]
\text{when}\quad a_1\in L^{p_1}_{(\omega
_1)}(\rr {2d}),\quad \text{and}\quad a_2\in L^{p_2}_{(\omega _2)}(\rr
{2d}).
\end{multline*}
\end{thm}

\par

\begin{proof}
From the assumptions it follows that at most one of $p_1$ and $p_2$
are equal to $\infty$. By reasons of symmetry we may therefore assume
that $p_2<\infty$.

\par

Since $W^2_{(\omega )}=M^2_{(\omega )}=L^2_{(\omega )}$ when $\omega
(X,Y)=\omega (X)$, in view of Theorem 2.2 in \cite{To5}, the result
follows from Theorem \ref{algthm2} in the case $p_1=p_2=p=2$.

\par

Now assume that $1/p_1+1/p_2-1/p=1$, $a_1\in
L^{p_1}(\rr {2d})$ and that $a_2\in \mathscr S(\rr {2d})$. Then
$$
\nm {a_1*_\sigma a_2}{L^p_{(\omega _0)}}\le (2/\pi )^{d/2}\nm {\,
|a_1|* |a_2|\, }{L^p_{(\omega _0)}} \le C \nm {a_1}{L^{p_1}_{(\omega
_1)}}\nm {a_2}{L^{p_2}_{(\omega _2)}},
$$
by Young's inequality. The result now follows in this case from
the fact that $\mathscr S$ is dense in $L^{p_2}_{(\omega _2)}$, when
$p_2<\infty$.

\par

The result now follows in the general case by multi-linear
interpolation between the case $p_1=p_2=p=2$ and the case
$1/p_1+1/p_2-1/p=1$, using Theorem 4.4.1 in \cite{BeLo} and the fact
that
$$
(L^{p_1}_{(\omega )}(\rr {2d}),(L^{p_2}_{(\omega )}(\rr
{2d}))_{[\theta ]} = L^{p_0}_{(\omega )}(\rr {2d}),
$$
when
$$
\frac {1-\theta}{p_1}+\frac \theta {p_2} = \frac 1{p_0}.
$$
(Cf. Chapter 5 in \cite{BeLo}.) The proof is complete.
\end{proof}

\par

By letting $p_1=p$ and $p_2=q\le \min (p,p')$, or $p_2=p$ and
$p_1=q\le \min (p,p')$, Theorem \ref{twistedleb} takes the following
form:

\par

\begin{cor}\label{twistedlebcor}
Assume that $\omega _0,\omega _1,\omega _2\in
\mathscr P(\rr {2d})$ and $p, q \in [1,\infty]$ satisfy
$$
\omega _0(X_1+X_2)\le C\omega _1(X_1)\omega _2(X_2),\quad
\text{and}\quad q\le \min (p,p')
$$
for some constant $C$. Then the map \eqref{twistmap} extends
uniquely to a continuous mapping from $L^p_{(\omega _1)}(\rr
{2d})\times L^q_{(\omega _2)}(\rr {2d})$ or $L^q_{(\omega _1)}(\rr
{2d})\times L^p_{(\omega _2)}(\rr {2d})$ to $L^p_{(\omega _0)}(\rr
{2d})$. Furthermore, for some constant $C$ it holds
\begin{alignat*}{2}
\nm {a_1*_\sigma a_2}{L^p_{(\omega _0)}} &\le C \nm {a_1}{L^p_{(\omega
_1)}} \nm {a_2}{L^q_{(\omega _2)}},& \  a_1&\in L^p_{(\omega
_1)}(\rr {2d}),\ a_2\in L^q_{(\omega _2)}(\rr {2d})
\intertext{and}
\nm {a_1*_\sigma a_2}{L^p_{(\omega _0)}} &\le C \nm {a_1}{L^q_{(\omega
_1)}} \nm {a_2}{L^p_{(\omega _2)}},& \  a_1&\in L^q_{(\omega
_1)}(\rr {2d}),\ a_2\in L^p_{(\omega _2)}(\rr {2d}).
\end{alignat*}
\end{cor}

\par

In the next section we need the following refinement of Theorem
\ref{twistedleb} concerning mixed Lebesgue spaces.

\par

\renewcommand{\rubrik}{Theorem \ref{twistedleb}$'$}

\begin{tom}
Assume that $k\in \{ 1,2\}$, $\omega _0,\omega _1,\omega _2\in
\mathscr P(\rr {2d})$ and $p,p_j, q,q_j \in [1,\infty]$ for $j=1,2$
satisfy
\begin{align*}
\omega _0(X_1+X_2) &\le C\omega _1(X_1)\omega _2(X_2),\quad p_1,p_2\le
p,\quad q_1,q_2\le q,
\\[1ex]
\max \Big ( \frac 1p,\frac 1{p'},\frac 1q,\frac 1{q'}\Big ) &\le \frac
1{p_1}+\frac 1{p_2}-\frac 1{p}\le 1\quad \text{and}
\\[1ex]
\max \Big ( \frac 1p,\frac 1{p'}, \frac 1q,\frac 1{q'}\Big ) &\le
\frac 1{q_1}+\frac 1{q_2}-\frac 1{q}\le 1,
\end{align*}
for some constant $C$. Then the map \eqref{twistmap} extends
uniquely to a continuous mapping from $L^{p_1,q_1}_{k,(\omega _1)}(\rr
{2d})\times L^{p_2,q_2}_{k,(\omega _2)}(\rr {2d})$ to
$L^{p,q}_{k,(\omega _0)}(\rr {2d})$. Furthermore, for some constant
$C$ it holds
\begin{multline*}
\nm {a_1*_\sigma a_2}{L^{p,q}_{k,(\omega _0)}} \le C \nm
{a_1}{L^{p_1,q_1}_{k, (\omega _1)}} \nm {a_2}{L^{p_2,q_2}_{k,(\omega
_2)}},
\\[1ex]
\text{when}\quad a_1\in L^{p_1,q_1}_{k,(\omega
_1)}(\rr {2d}),\quad \text{and}\quad a_2\in L^{p_2,q_2}_{k,(\omega
_2)}(\rr {2d}).
\end{multline*}
\end{tom}

\par

\begin{proof}
The result follows from Minkowski's inequality when $p_1=q_1=1$ and
when $p_2=q_2=1$. Furthermore, the result follows in the case
$p_1=p_2=q_1=q_2=2$ from Theorem \ref{twistedleb}. In the general
case, the result follows from these cases and multi-linear
interpolation.
\end{proof}

\par

%%%%%%%%%%%%%%%%%%%%%%%%%%%%%%%
\section{Window functions in modulation space norms}\label{sec2.5}
%%%%%%%%%%%%%%%%%%%%%%%%%%%%%%%

\par

In this section we use the results in the previous section to prove
that the class of permitted windows in the modulation space norm can
be enlarged. More precisely we have the following.

\par

\begin{prop}\label{possiblewindows}
Assume that $p,p_0,q,q_0\in [1,\infty ]$ and $\omega ,v\in \mathscr
P(\rr {2d})$ are such that $p_0,q_0\le \min (p,p',q,q')$, $\check v=v$ and $\omega$
is $v$-moderate. Also assume that $f\in
\mathscr S'(\rr d)$. Then the following is true:
\begin{enumerate}
\item if $\fy \in M^{p_0,q_0}_{(v)}(\rr d)\setminus 0$, then $f\in
M^{p,q}_{(\omega )}(\rr d)$ if and only if $V_\fy f\in
L^{p,q}_{1,(\omega )}(\rr {2d})$. Furthermore, $\nmm f\equiv \nm
{V_\fy f}{L^{p,q}_{1,(\omega )}}$ defines a norm for $M^{p,q}_{(\omega
)}(\rr d)$, and different choices of $\fy$ give rise to equivalent
norms;

\vrum

\item if $\fy \in W^{p_0,q_0}_{(v)}(\rr d)\setminus 0$, then $f\in
W^{p,q}_{(\omega )}(\rr d)$ if and only if $V_\fy f\in
L^{p,q}_{2,(\omega )}(\rr {2d})$. Furthermore, $\nmm f\equiv \nm
{V_\fy f}{L^{p,q}_{2,(\omega )}}$ defines a norm for $W^{p,q}_{(\omega
)}(\rr d)$, and different choices of $\fy$ give rise to equivalent
norms.
\end{enumerate}
\end{prop} 

\par

For the proof we recall that the (cross) Wigner distribution of $f\in
\mathscr S'(\rr d)$ and $g\in \mathscr S'(\rr d)$ is defined by the
formula
$$
W_{f,g}(x,\xi )=\mathscr F (f(x/2-\cdo )\overline {g(x/2+\cdo
)}(\xi ).
$$
If $f,g\in \mathscr S(\rr d)$, then $W_{f,g}$ takes the form
$$
W_{f,g}(x,\xi )=(2\pi )^{-d/2}\int
f(x-y/2)\overline{g(x+y/2)}e^{i\scal y\xi}\, dy.
$$
From the definitions it follows that
$$
W_{f,g}(x,\xi )=2^de^{i\scal x\xi /2}V_{\check g}f(2x,2\xi ),
$$
which implies that
\begin{equation}\label{stftwiennorms}
\nm {W_{f,\check \fy}}{L^{p,q}_{k,(\omega _0)}} = 2^d\nm {V_\fy
f}{L^{p,q}_{k,(\omega )}},\quad \text{when}\quad \omega _0(x,\xi
)=\omega (2x,2\xi )
\end{equation}
for $k=1,2$.

\par

Finally, by Fourier's inversion formula it follows that if $f_1,g_2\in
\mathscr S'(\rr d)$ and $f_1,g_2\in L^2(\rr d)$, then
\begin{equation}\label{wigntwconv}
W_{f_1,g_1}*_\sigma W_{f_2,g_2}  = (\check f_2,g_1)_{L^2}W_{f_1,g_2}.
\end{equation}

\par

\begin{proof}[Proof of Theorem \ref {possiblewindows}.]
We may assume that $p_0=q_0=\min (p,p',q,q')$. Assume that $\fy ,\psi \in
M^{p_0,q_0}_{(v)}(\rr d)\subseteq L^2(\rr d)$, where the inclusion follows from the fact that $p_0,q_0\le 2$ and $v\ge c$ for some constant $c>0$. Since $\omega$
is both $v$-moderate and $\check v$-moderate, and $\nm {V_\fy \psi }
{L^{p_0,q_0}_{k,(v)}}=\nm {V_\psi \fy}{L^{p_0,q_0}_{k,(v)}}$ when $\check v=v$, the
result follows if we prove that
\begin{equation}\label{eststft3}
\nm {V_\fy f}{L^{p,q}_{k,(\omega )}}\le C\nm {V_\psi
f}{L^{p,q}_{k,(\omega )}}\nm {V_\fy \psi }{L^{p_0,q_0}_{k,(v)}},
\end{equation}
for some constant $C$ which is independent of $f\in \mathscr S'(\rr
d)$ and $\fy ,\psi \in M^{p_0,q_0}_{(v)}(\rr d)$.

\par

If $p_1=p$, $p_2=p_0$, $q_1=q$, $q_2=q_0$, $\omega _0=\omega
(2\cdo )$ and $v_0=v(2\cdo )$, then Theorem
\ref{twistedleb}$'$ and \eqref{wigntwconv} give
\begin{multline*}\label{wigntwconvnorm}
\nm {V_\fy f}{L^{p,q}_{k,(\omega )}} = C_1\nm
{W_{f,\check \fy}}{L^{p,q}_{k,(\omega _0)}}
\\[1ex]
=C_2\nm {W_{f,\check \psi}*_\sigma W_{\psi,\check
\fy}}{L^{p,q}_{k,(\omega _0)}}\le C_3 \nm {W_{f,\check \psi}}
{L^{p,q}_{k,(\omega _0)}}\nm {W_{\psi ,\check \fy}}{L^{p_0,q_0}_{k,(v_0)}}
\\[1ex]
=C_4 \nm {V_\psi f}{L^{p,q}_{k,(\omega )}}\nm {V_\fy \psi }{L^{p_0,q_0}_{k,(v)}},
\end{multline*}
and \eqref{eststft3} follows. The proof is complete.
\end{proof}

\par

%%%%%%%%%%%%%%%%%%%%%%%%%%%%%%%
\section{Pseudo-differential operators
with symbols in modulation
spaces}\label{sec3}
%%%%%%%%%%%%%%%%%%%%%%%%%%%%%%%

\par

In this section we discuss continuity of pseudo-differential operators
with symbols in modulation spaces of the form $W^{p,q}_{(\omega )}$, when acting between
modulation spaces.

\par

The main result is the following theorem, which is based on
Proposition \ref{cordokodjP3.3}.

\par

\begin{thm}\label{Wpseudos}
Assume that $p,q\in [1,\infty ]$, $\omega _1,\omega _2\in \mathscr
P(\rr {2d})$ and $\omega \in \mathscr P(\rr {4d})$ are such that
\begin{equation}\label{omegascond}
\frac {\omega _2(x,\xi +\eta )}{\omega _1(x+y,\xi )} \le C\omega
(x,\xi ,\eta ,y)
\end{equation}
holds for some constant $C$ which is independent of $x,y,\xi ,\eta \in
\rr d$. Also assume that $a\in W^{q,p}_{(\omega )}(\rr {2d})$. Then
the definition of $a(x,D)$ from $\mathscr S(\rr d)$ to $\mathscr
S'(\rr d)$ extends uniquely to a continuous mapping from
$M^{p',q'}_{(\omega _1)}(\rr d)$ to $W^{q,p}_{(\omega _2)}(\rr
d)$. Furthermore, it holds
$$
\nm {a(x,D)f}{W^{q,p}_{(\omega _2)}} \le C\nm {a}{W^{q,p}_{(\omega
)}}\nm f{M^{p',q'}_{(\omega _1)}}
$$
for some constant $C$ which is independent of $f\in M^{p',q'}_{(\omega
_1)}(\rr d)$ and $a\in W^{q,p}_{(\omega )}(\rr {2d})$.
\end{thm}

\par

The proof is based on duality, using the fact that if $a\in \mathscr S'(\rr {2d})$, $f,g\in \mathscr S(\rr d)$ and $T$ is the operator, defined by
\begin{align}
(T\psi )(x,\xi ) &=\psi (\xi ,-x)\label{pseudoSTFT1}
\intertext{when $\psi \in \mathscr S'(\rr {2d})$, then}
(a(x,D)f,g)_{L^2(\rr d)} &= (2\pi )^{-d/2}(T(\mathscr F
a),V_fg)_{L^2(\rr {2d})},\label{pseudoSTFT2}
\end{align}
by Fourier's inversion formula.

\par

For the proof of  Theorem \ref{Wpseudos} we shall combine \eqref{pseudoSTFT2} with the following weighted version of Proposition 3.3 in \cite{CO1}.

\par

\begin{prop}\label{cordokodjP3.3}
Assume that $f_1,f_2\in \mathscr S'(\rr d)$, $p,q\in [1,\infty
]$, $\omega _0\in \mathscr P(\rr {4d})$ and $\omega _1,\omega
_2\in \mathscr P(\rr {2d})$. Also assume that $\fy _1,\fy _2\in
\mathscr S(\rr d)$, and let $\Psi =V_{\fy _1}\fy
_2$. Then the following is true:
\begin{enumerate}
\item if
\begin{equation}\label{omegarel1}
\omega _0(x,\xi ,\eta ,y)\le C\omega
_1(-x-y,\eta ) \omega _2(-y,\xi +\eta )
\end{equation}
for some constant $C$, then
$$
\nm {V_\Psi (V_{f_1}f_2)}{L^{p,q}_{1,(\omega _0)}} \le C\nm
{V_{\fy _1}f_1}{L^{p,q}_{1,(\omega _1)}}\nm {V_{\fy
_2}f_2}{L^{q,p}_{2,(\omega _2)}}\text ;
$$

\vrum

\item if
\begin{equation}\label{omegarel2}
\omega _1(-x-y,\eta ) \omega _2(-y,\xi +\eta )\le C\omega _0(x,\xi
,\eta ,y)
\end{equation}
for some constant $C$, then
$$
\nm {V_{\fy _1}f_1}{L^{p,q}_{1,(\omega _1)}}\nm {V_{\fy
_2}f_2}{L^{q,p}_{2,(\omega _2)}}\le C\nm {V_\Psi
(V_{f_1}f_2)}{L^{p,q}_{1,(\omega _0)}}\text ;
$$

\vrum

\item if \eqref{omegarel1} and \eqref{omegarel2} hold for some constant $C$, then $f_1\in
M^{p,q}_{(\omega _1)}(\rr d)$ and $f_2\in W^{q,p}_{(\omega _2)}(\rr
d)$, if and only if $V_{f_1}f_2\in M^{p,q}_{(\omega _0)}(\rr {2d})$,
and
$$
C^{-1}\nm {V_{f_1}f_2}{M^{p,q}_{(\omega _0)}}\le \nm
{f_1}{M^{p,q}_{(\omega _1)}}\nm {f_2}{W^{q,p}_{(\omega _2)}}
\le C\nm {V_{f_1}f_2}{M^{p,q}_{(\omega _0)}},
$$
for some constant $C$ which is independent of $f_1$ and $f_2$.
\end{enumerate}
\end{prop}

\par

\begin{proof}
It suffices to prove (1) and (2), and then we prove only
(1), since (2) follows by similar arguments. We shall mainly follow
the proof of Proposition 3.3 in \cite{CO1}, and then we only prove
the result in the case $p<\infty$ and $q<\infty$. The small
modifications when $p=\infty$ or $q=\infty$ are left for the reader.

\par

By Fourier's inversion formula we have
$$
|V_\fy f_1(-x-y,\eta )V_\fy f_2(-y,\xi +\eta
)|=|V_{\Psi}(V_{f_1}f_2)(x,\xi ,\eta ,y)|
$$
(cf. e.{\,}g. \cite{CO1,Fo1,Groc2,To0,To1}). Hence, if
$$
F_1(x,\xi ) = V_{\fy _1} f_1(x,\xi )\omega _1(x,\xi ) \quad
\text{and}\quad F_2(x,\xi ) = V_{\fy _2} f_2(x,\xi )\omega _2(x,\xi ),
$$
then we get
\begin{multline*}
\nm {V_\Psi (V_{f_1}f_2)}{L^{p,q}_{1,(\omega _0)}}^q
\\[1ex]
=\iint _{\rr {2d}}\Big
(\iint _{\rr {2d}}|V_\Psi (V_{f_1}f_2)(x,\xi ,\eta ,y)\omega _0(x,\xi
,\eta ,y)|^p\, dxd\xi \Big )^{q/p}\, dyd\eta 
\\[1ex]
\le C^q \iint _{\rr {2d}}\Big (\iint _{\rr {2d}}|F_1(-x-y,\eta )
F_2(-y,\xi+\eta ))|^p\, dxd\xi \Big )^{q/p}\, dyd\eta .
\end{multline*}
By taking $-y$, $\xi +\eta$, $-x-y$ and $\eta$ as new variables of
integration, we obtain
\begin{multline*}
\nm {V_\Psi (V_{f_1}f_2)}{L^{p,q}_{1,(\omega _0)}}^q
\\[1ex]
\le C^q\Big (\int _{\rr
d}\Big ( \int _{\rr d} |F_1(x,\eta )|^p\, dx
\Big )^{q/p}\, d\eta \Big )\Big ( \int _{\rr d}\Big (
\int _{\rr d} |F_2(y,\xi )|^p\, d\xi \Big
)^{q/p}\, dy \Big )
\\[1ex]
=C^q\nm {V_{\fy _1}f_1}{L^{p,q}_{1,(\omega _1)}}^q\nm {V_{\fy
_2}f_2}{L^{q,p}_{2,(\omega _2)}}^q.
\end{multline*}
This proves the assertion.
\end{proof}

\par

\begin{proof}[Proof of Theorem \ref{Wpseudos}.]
We may assume that \eqref{omegascond} holds for $C=1$ and with
equality. We start to prove the result when $1<p$ and $1<q$.
Let
$$
\omega _0(x,\xi ,\eta ,y ) = \omega (-y,\eta ,\xi ,-x)^{-1},
$$
and assume that $a\in W^{q,p}_{(\omega )}(\rr {2d})$ and $f,g\in
\mathscr S(\rr d)$. Then $a(x,D)f$ makes sense as an element in
$\mathscr S'(\rr d)$. 

\par

By Proposition \ref{cordokodjP3.3}  we get
\begin{equation}\label{Vfgigen}
\nm {V_fg}{M^{p',q'}_{(\omega _0)}}\le C\nm f{M^{p',q'}_{(\omega
_1)}}\nm g{W^{q',p'}_{(\omega _2^{-1})}}.
\end{equation}

\par

Furthermore, if $T$ is the same as in \eqref{pseudoSTFT1}, then it
follows by Fourier's inversion formula that
$$
(V_{\fy }(T\widehat a))(x,\xi ,\eta ,y )=e^{-i(\scal x\eta +\scal y\xi )}(V_{T\widehat \fy }a)(-y,\eta ,\xi
,-x).
$$
This gives
\begin{multline*}
|(V_{\fy}(T\widehat a))(x,\xi ,\eta ,y )\omega _0(x,\xi ,\eta
,y)^{-1}|
\\[1ex]
=|(V_{\fy _1}a)(-y,\eta ,\xi ,-x)\omega (-y,\eta ,\xi ,-x)|,
\end{multline*}
when $\fy _1=T\widehat \fy$. Hence, by applying the $L^{p,q}_1$ norm we obtain $\nm {T\widehat
a}{M^{p,q}_{(\omega _0^{-1})}}=\nm a{W^{q,p}_{(\omega )}}$.

\par

It  now follows from (12) and \eqref{Vfgigen} that
\begin{multline}\label{estimates1}
|(a(x,D)f,g)|=(2\pi )^{-d/2}|(T\widehat a,V_{\overline g}\overline
f)|
\\[1ex]
\le C_1\nm {T\widehat a}{M^{p,q}_{(\omega _0^{-1})}}\nm
{V_{f}g}{M^{p',q'}_{(\omega _0)}}
\\[1ex]
\le C_2 \nm a{W^{q,p}_{(\omega )}}\nm f{M^{p',q'}_{(\omega _1)}}\nm
g{W^{q',p'}_{(\omega _2^{-1})}}.
\end{multline}
The result now follows by the facts that $\mathscr S(\rr d)$ is dense
in $M^{p',q'}_{(\omega _1)}(\rr d)$, and that the dual of
$W^{q',p'}_{(\omega _2^{-1})}$ is $W^{q,p}_{(\omega _2)}$ when
$p,q>1$.

\par

If instead $p=1$ and $q<\infty$, or  $q=1$ and $p<\infty$, then we
assume that $f\in M^{p',q'}_{(\omega _1)}$ and $a\in \mathscr S(\rr
{2d})$. Then $a(x,D)f$ makes sense as an element in $\mathscr S(\rr
d)$, and from the first part of the proof it follows that
\eqref{estimates1} still holds. The result now follows by duality and
the fact that $\mathscr S(\rr {2d})$ is dense in $W^{q,p}_{(\omega
)}(\rr {2d})$ for such choices of $p$ and $q$.

\par

It remains to consider the cases $p=q'=1$ and $p=q'=\infty$. In this
case, the result follows by using the fact that $\mathscr S$ is dense
in $M^{\infty ,1}_{(\omega )}$ and $W^{1,\infty}_{(\omega )}$ with
respect to the narrow convergence. The proof is complete.
\end{proof}

\end{document}